\documentclass{amsart}

\usepackage{amsmath,amsxtra,amsthm,amssymb,amsfonts,graphicx} 

\usepackage[pdfusetitle,bookmarks=true,bookmarksnumbered=false,bookmarksopen=false,breaklinks=false,pdfborder={0 0 0},backref=false,colorlinks=true]{hyperref}

\usepackage[latin1]{inputenc}
\usepackage{cite}
\usepackage{setspace}
\usepackage{amssymb}
\usepackage{amsthm}

\usepackage{empheq}
\usepackage{graphics}
\graphicspath{{Figs/}}
\usepackage{graphicx}
\usepackage[margin=3cm]{geometry}

\usepackage{array}

\usepackage{color}
\usepackage{verbatim}
 
\newcommand{\dom}{\ensuremath{\operatorname{dom}}}
\newcommand{\prox}{\ensuremath{\operatorname{prox}}}

\newcommand{\gammaTo}{\ensuremath{\overset{\Gamma}{\to}}}
\newcommand{\moscoTo}{\ensuremath{\overset{\textrm{M}}{\to}}}
\newcommand{\weakto}{\ensuremath{\stackrel{w}{\to}}}

\newcommand{\mR}{\ensuremath{\mathbb{R}}}
\newcommand{\mN}{\ensuremath{\mathbb{N}}}

\newcommand{\HH}{\mathcal{H}}

\newcommand{\argmin}{\mathop{\rm argmin}}
\newcommand{\di}{d}

\newcommand*{\trez}[1]{\tfrac{1}{#1}}     

\newcommand{\fun}{\colon\HH\to(-\infty,+\infty]}

\newtheorem{theorem}{Theorem}[section]

\newtheorem{lemma}[theorem]{Lemma}
\newtheorem{remark}[theorem]{Remark}

\newtheorem{corollary}[theorem]{Corollary}

\begin{document}
\title{Convergence of Functions and their Moreau Envelopes on Hadamard Spaces}
\author{Miroslav Ba{\v c}{\'a}k}
\address{Max Planck Institute for Mathematics in the Sciences, Inselstr. 22, 04 103 Leipzig, Germany}
\email{bacak@mis.mpg.de}

\author{Martin Montag}
\address{University of Kaiserslautern, Department of Mathematics, Kaiserslautern and Fraunhofer ITWM, Kaiserslautern, Germany}
\email{m.j.montag@posteo.de}

\author{Gabriele Steidl}
\address{University of Kaiserslautern, Department of Mathematics, Kaiserslautern and Fraunhofer ITWM, Kaiserslautern, Germany}
\email{steidl@mathematik.uni-kl.de}


\hyphenation{weight-ed Hada-mard}
  
\date{}

\begin{abstract}
A well known result of H. Attouch states that the Mosco convergence
of a sequence of proper convex lower semicontinuous functions
defined on a Hilbert space is equivalent to the pointwise convergence
of the associated Moreau envelopes.
In the present paper we generalize this result to Hadamard spaces.
More precisely, while it has already been known that the Mosco convergence of a sequence 
of convex lower semicontinuous functions on a Hadamard space implies 
the pointwise convergence of the corresponding Moreau envelopes, 
the converse implication was an open question. We now fill this gap.  

Our result has several consequences. It implies, for instance, the equivalence of the Mosco and Frol\'{i}k-Wijsman convergences of convex sets.
As another application, we show that there exists a~complete metric on the cone of proper convex lower semicontinuous 
functions on a separable Hadamard space such that a~sequence of functions converges in this metric if and only if it converges in the sense of Mosco.
\end{abstract}

\maketitle

\section{Introduction} \label{sec:intro}
%
Proximal mappings and Moreau envelopes of convex functions play a central role
in convex analysis and moreover they appear in various minimization algorithms which have recently 
found application in, for instance, signal/image processing and machine learning. 
For overviews, see for instance \cite{BC11,BSS2014,combettespesquet,PB2013} 
and for proximal algorithms on Hadamard manifolds, e.g., \cite{AF05,FO02,QO09,BFO15a}.

In the present paper we are concerned with the relation between Moreau envelopes and the Mosco convergence. 
Specifically, a well known result of H.~Attouch says that the Mosco convergence of a sequence of 
convex lower semicontinuous  functions on a Hilbert space 
is completely characterized by the pointwise convergence of their Moreau envelopes \cite[Theorem 1.2]{Attouch79}. 
Note that this result was later on extended into a certain class of Banach spaces \cite[Theorem~3.26]{Attouch84}. 

We briefly recall the result in a Hilbert space $H$. To this end we first fix our notation and terminology. 
A basic reference on variational analysis is~\cite{RW2009}. 
The \emph{domain} of a function $f\colon H \to (-\infty,+\infty]$ is the set $\dom f := \{ x \in H\colon f(x) < +\infty\}$ and $f$ is \emph{proper} if $\dom f \not = \emptyset$. A function $f\colon H \to (-\infty,+\infty]$ is called \emph{lower semicontinuous} (lsc) if the level sets $\{x \in H\colon f(x) \le \alpha\}$ are closed for every $\alpha \in \mR$. Let $\Gamma_0(H)$ denote the cone of proper convex lsc functions defined on a Hilbert space~$H$. For a proper convex lsc function $f\in \Gamma_0(H)$ and $\lambda>0,$ the \emph{proximal mapping} $J_{\lambda}^f \colon H \to H$ is defined by
\begin{equation} \label{eq:prox}
	J_{\lambda}^f(x) = \argmin_{y\in H}
		 \Bigl\{ f(y) + \tfrac{1}{2\lambda} d(x,y)^2\Bigr\}, \qquad x\in H,
\end{equation}
where $d(x,y) := \|x-y\|$. Indeed, the minimizer of the right-hand side exists and is unique. For a~fixed~$f,$ we will write just $J_\lambda$ instead of $J_{\lambda}^f$. We also mention that an alternative symbol $\prox_{\lambda f}$ for the proximal mapping appears frequently in the literature, but we will not use it in the present paper.

Given $\lambda > 0$, the \emph{Moreau envelope} of $f$ is defined by
\begin{equation}		\label{eq:moreau}
  f_\lambda (x) := \min_{y\in H} \Bigl\{ f(y) + \tfrac{1}{2\lambda} d(x,y)^2 \Bigr\}, \qquad x\in H.
\end{equation}
The following definition goes back to U.~Mosco~\cite{Mosco69}. A sequence $\{f_n\}_{n}$ of functions $f_n \colon H\to (-\infty,+\infty]$ \emph{Mosco-converges} to $f \colon  H \to (-\infty,+\infty]$, abbreviated $ f_n \moscoTo f$,
if, for each $x \in H$, the following two conditions are fulfilled:
\begin{enumerate}
 \item[i)] $f(x) \leq \liminf_{n \to \infty} f_n(x_n)$ whenever  $x_n \weakto x$, 
 \item[ii)] there is a sequence $\{y_n\}_n$ such that $y_n \to x $ and $f_n(y_n) \to f(x),$
\end{enumerate}
where $x_n \weakto x$ stands for weak convergence. There is a weaker type of convergence, known as $\Gamma$-convergence, in which we just replace the first statement in the above definition by
\begin{enumerate}
 \item[i)] $f(x) \leq \liminf_{n \to \infty} f_n(x_n)$, whenever  $x_n \to x$.
\end{enumerate}
We refer the reader to~\cite{Braides02,maso} for further information on $\Gamma$-convergence. Obviously, $\Gamma$-convergence can be defined on an arbitrary topological space. It is a suitable notion in the study of minimization problems because minimizers of the limit function~$f$
are related to approximate minimizers of the functions $f_n$,
see~\cite[Theorem~1.21]{Braides02}, or, in general topological spaces~\cite[Theorem~1.10]{Attouch84}.  
The notion of $\tau$-epi convergence~\cite[Definition~1.9]{Attouch84} is just another name for $\Gamma$-convergence; see \cite[Proposition~1.18]{Braides02}.

We are now able to present the promised result due to H.~Attouch; see \cite[Theorem 1.2]{Attouch79} and \cite[Theorem~3.26]{Attouch84}.
\begin{theorem}  \label{thm:Attouch-MoreauMosco}
Let $H$ be a Hilbert space and $\left\{f_n\right\}_n \subset\Gamma_0(H)$.
Then the following statements are equivalent:
\begin{itemize}
\item[{\rm i)}] The sequence $\{f_n\}_n$ converges to a proper function $f\colon H\to (-\infty,+\infty]$ 
  in the sense of Mosco. 
\item[{\rm ii)}] For each $\lambda >0$, the sequence of Moreau envelopes 
  $\{ f_{n,\lambda} \}_n$ of $\{ f_n \}_n$ converges pointwise
  to the Moreau envelope $f_\lambda$ of a proper convex lsc function $f$.
\end{itemize}
\end{theorem}
%
The aim of the present paper is to generalize Theorem \ref{thm:Attouch-MoreauMosco} to Hadamard spaces. 
Note that both $\Gamma$- and Mosco convergences have already been used in this framework. In~\cite{jost98}, 
J.~Jost studied harmonic mappings with metric space targets and as a tool he introduced $\Gamma$-convergence on Hadamard spaces. 
He also defined the Mosco convergence by saying that a sequence of convex lsc functions on a Hadamard space Mosco converges if their Moreau envelopes 
converge pointwise~\cite{jost98}. 
In~\cite{kuwae}, K.~Kuwae and T.~Shioya studied both $\Gamma$- and Mosco convergence in Hadamard space in depth and obtained numerous generalizations. 
They have already given the standard definition of the Mosco convergence used in this paper (relying on the notion of weak convergence)   
and right after their Definition 5.7 in \cite{kuwae} they note ``Jost's definition of Mosco convergence\ldots seems unfitting in view of Mosco's original definition.'' 
By our main result (Theorem \ref{thm:moreau2mosco}) it follows that both definitions are equivalent. 

In \cite[Proposition 5.12]{kuwae}, the authors prove that the Mosco convergence of \emph{nonnegative} 
convex lsc functions on a Hadamard space implies the pointwise convergence of their Moreau envelopes. 
This result was later proved in \cite[Theorem 4.1]{bac15} without the nonnegativity assumption. The inverse implication was left open; 
see \cite[Question 5.2.5]{b2014}. In the present note we answer this question in the positive. As a corollary of our main result 
we obtain that the Mosco convergence of convex closed sets is equivalent to the Frol\'{i}k-Wijsman convergence.

In \cite{bac15,kuwae} the Mosco convergence of functions on Hadamard spaces was studied in connection with \emph{gradient flows.} 
In particular, it was shown in \cite{bac15} that 
the Mosco convergence of convex lsc functions on a Hadamard space implies the pointwise convergence of the associated gradient flow semigroups. 
Interestingly, apart from applications of Hadamard space gradient flows into harmonic mappings theory, see e.g., \cite{jost98},\cite[Section 8]{sturm}, 
there have been also other motivations. 
Most remarkably, gradient flows of a convex function on a Hadamard space appear as an important tool in K\"ahler geometry 
in connection with Donaldson's conjecture on Calabi flows \cite{darvas,streets}. It has also similarly inspired new developments in Riemannian geometry~\cite{gursky}. 
Finally, in \cite{bacak-kovalev}, a gradient flow of a convex continuous function was used to construct a Lipschitz retraction in a Hadamard space. 
For discrete-time gradient flows of convex functions in Hadamard spaces and their applications, see \cite{bac14,b2014}. 

Mosco convergence was proved in \cite{BER2015} in connection with an image processing task for a sequence of functions defined on refined image grids.

If the Hilbert space~$H$ is \emph{separable,} a nice consequence of Theorem \ref{thm:Attouch-MoreauMosco} 
is the existence of a complete separable metric~$\rho$ on~$\Gamma_0(H)$ such that $\rho\left(f,f_n \right)\to0$ if and only if $f_n$ converges to~$f$ 
in the sense of Mosco; see \cite[Theorem 3.36]{Attouch84}. 
Note that this result was shown even for reflexive separable Banach spaces. For its applications into stochastic optimization, the interested reader is referred to~\cite{attouch-wets}. 
In Section~\ref{sec:topo} we prove that there exists a complete metric on~$\Gamma_0(\HH),$ where $\HH$ is a separable Hadamard space now, which corresponds to the Mosco convergence.

\section{Preliminaries on Hadamard spaces} \label{sec:notation}
%
First we collect the preliminaries on Hadamard spaces required for our proof. For the details, we refer to \cite{b2014}. 
A complete metric space $(\HH,d)$ is called a \emph{Hadamard space} if it is geodesic and the following condition holds
\begin{equation}\label{eq:reshet0}
d(x,v)^2 + d(y,w)^2 \le d(x,w)^2 + d(y,v)^2 + 2d(x,y)d(v,w),
\end{equation}
for every $x,y,v,w \in \HH$. 
Recall that a metric space $(X,d)$ is \emph{geodesic} if every two points $x,y \in X$  are connected by a geodesic, 
that is, there exists a curve $\gamma_{\overset{\frown}{x,y}}\colon [0,1] \to X$ such that
\begin{equation*}
  d \bigl(\gamma_{\overset{\frown}{x,y}}(s),\gamma_{\overset{\frown}{x,y}}(t) \bigr) 
  = \lvert s - t\rvert d \bigl(\gamma_{\overset{\frown}{x,y}}(0),\gamma_{\overset{\frown}{x,y}}(1) \bigr), \qquad \text{for every } s,t \in [0,1],
\end{equation*}
and $\gamma_{\overset{\frown}{x,y}}(0)=x$ and $\gamma_{\overset{\frown}{x,y}}(1)=y.$
Inequality~\eqref{eq:reshet0} expresses the fact that Hadamard spaces have nonpositive curvature 
and in particular it implies that geodesics are uniquely determined by their endpoints. It also yields a (formally weaker) inequality
\begin{equation}\label{eq:reshet}
d(x,v)^2 + d(y,w)^2 \le d(x,w)^2 + d(y,v)^2 + d(x,y)^2 + d(v,w)^2,
\end{equation}
valid for every $x,y,v,w \in \HH$. 

The class of Hadamard spaces comprises Hilbert spaces, Hadamard manifolds (that is, complete simply connected Riemannian manifolds of nonpositive sectional curvature), Euclidean buildings and CAT(0) complexes.

The definition of proper and lsc functions carries over from the Hilbert space setting. 
A function $f\fun$ is called \emph{convex} if for every $x,y \in \HH$ the function 
$f \circ \gamma_{\overset{\frown}{x,y}}$ is
convex, i.e., if
\begin{equation} \label{eq:conv}
	f\bigl( \gamma_{\overset{\frown}{x,y}}(t) \bigr)
	\le (1-t) f\bigl( \gamma_{\overset{\frown}{x,y}} (0) \bigr)
	+ t f \bigl(\gamma_{\overset{\frown}{x,y}}(1)\bigr),
\end{equation}
for each $t \in [0,1]$. Let $\Gamma_0(\HH)$  denote the convex cone of proper convex lsc functions on $\HH$.
For $f \in \Gamma_0(\HH)$ and $\lambda > 0$, the \emph{proximal mapping} $J_{\lambda}^f \colon \HH \to \HH$ 
and the \emph{Moreau envelope} $f_\lambda: \HH \to \mR$
are defined as in \eqref{eq:prox} and \eqref{eq:moreau}, respectively,
where the Hilbert space distance has to be replaced by the Hadamard space metric. 
The proximal mapping is nonexpansive. 
The Moreau envelope is convex and continuous. 
Note that in \cite[p. 42]{b2014} it was incorrectly 
claimed that the Moreau envelope is not necessarily lsc. 
We now correct this statement. For a general result in metric spaces, we refer to \cite[Theorem 2.64]{Attouch84}.
Given a metric space $(X,d),$ we denote by $B_R(x)$ the closed ball with radius $R>0$ centered at a~point $x\in X$.
%
\begin{lemma} \label{lem:continuity-M-Y}
 For each $\lambda > 0,$ the Moreau envelope $f_\lambda$ of a function $f \in \Gamma_0(\HH)$ is locally Lipschitz on~$\HH.$
\end{lemma}
\begin{proof}
Let $x_0\in\HH$ and $R>0.$ Choose $C>0$ such that $d\left(z,J_\lambda z\right)<C$ for each $z\in B_R\left(x_0\right).$ Then for every $x,y\in B_R\left(x_0\right)$ we have
\begin{align*}
 f_\lambda(y) & \le f\left( J_\lambda x\right)+\trez{2\lambda}d\left(y,J_\lambda x\right)^2 \\ &\le f_\lambda(x) -\trez{2\lambda}d\left(x,J_\lambda x\right)^2+\trez{2\lambda}d\left(y,J_\lambda x\right)^2 \\ 
&=  f_{\lambda}(x)
+\trez{2\lambda} \bigl( \di(y, J_{\lambda}x) - \di(x, J_{\lambda}x)\bigr)
                            \bigl(\di(y, J_{\lambda}x) + \di(x, J_{\lambda}x)\bigr)\\
&\le f_\lambda(x) +\trez{2\lambda}d(x,y)\left(2d\left(x,J_\lambda x\right)+d(x,y)\right) \\ 
& \le f_\lambda(x) +\trez{2\lambda}d(x,y)\left(2C+2R\right),
\end{align*}
where we used the triangle inequality for the third inequality.
Interchanging $x$ and $y$ yields
\begin{equation*}
 \left| f_\lambda(x)-f_\lambda(y)\right| \le \frac{C+R}{\lambda}d(x,y). \qedhere
\end{equation*}
\end{proof}

The Moreau envelope of a function $f \in \Gamma_0(\HH)$ possesses the semigroup property
\begin{equation*}
(f_{\lambda})_\mu  = f_{\lambda + \mu} , \quad \lambda, \mu >0.
\end{equation*}
Furthermore, for each $x \in \HH$, we have
\begin{equation}\label{conv_lambda}
\lim_{\lambda \to +0} f_\lambda (x) = f(x),
\end{equation}
see \cite[Theorem 2.46]{Attouch84}. By
\begin{equation*}
 \lim_{\mu \to +0} f_{\lambda + \mu} (x) = \lim_{\mu \to +0} \left(f_\lambda \right)_\mu (x) = f_\lambda (x) 
\end{equation*}
the mapping $\lambda \mapsto f_\lambda(x)$ is right-continuous.
In fact it is even locally Lipschitz~\cite[Proposition 2.2.27]{b2014}.

We will also need the following two auxiliary lemmas.
\begin{lemma}
 Let $(\HH,d)$ be a Hadamard space and $f\in\Gamma_0(\HH).$ Then for every $x, y \in \HH$ and $\lambda > 0$, we have
 \begin{equation} \label{eq:ineq}
  f\left(J_\lambda  x\right) + \tfrac1{2\lambda} d\left(x,J_\lambda x\right)^2 
  + \tfrac1{2\lambda} d\left(J_\lambda  x,y\right)^2  \leq f(y)+\tfrac1{2\lambda} d(x,y)^2.
\end{equation}
\end{lemma}
\begin{proof}
The proof can be found in \cite[Lemma 2.2.23]{b2014}.
\end{proof}

\begin{lemma}		\label{thm:moreau-equiMinorization}
Let $\HH$ be a Hadamard space and $\{f_n\}_{n}\subset\Gamma_0(\HH).$ Suppose that for some $\lambda > 0$ and $x_0 \in \HH$ 
the Moreau envelopes fulfill $f_{n,\lambda} (x_0) \to f_0 \in\mR$ as $n\to\infty.$ Then there exists $r > 0$ such that
\begin{equation} 
  f_n (x) \geq -r \left( d(x,x_0)^2 + 1 \right),   \label{enum:moreau-equiMinorization}
\end{equation}
for each $x \in \HH.$
\end{lemma}
\begin{proof}
By assumptions, there is some $n_0$ such that 
$\left|f_{n,\lambda}(x_0) - f_0\right| \leq 1$ for each $n \geq n_0$. 
Thereby, from
\begin{equation*}
  f_{n,\lambda}(x_0) = \inf_x \bigl\{ f_n(x) + \tfrac{1}{2\lambda} d(x,x_0)^2 \bigr\}
\end{equation*}
we obtain that
\begin{equation}  \label{eq:moreau-minor}
  f_n(x) \geq f_{n,\lambda}(x_0) - \tfrac{1}{2\lambda} d(x,x_0)^2 \geq f_0 -1- \tfrac{1}{2\lambda} d(x,x_0)^2,
\end{equation}
for every $x \in \HH$ and $n\geq n_0.$ This yields the desired assertion.
\end{proof} 

The definition of Mosco convergence in Hadamard spaces requires a notion of weak convergence.
For a bounded sequence $\{x_n\}_n$ of points $x_n \in \HH$, the function 
$\omega \colon \HH \to [0, +\infty)$ defined by
\begin{equation}
  \omega(x;\, \{x_n\}_n) := \limsup_{n\to\infty} d(x, x_n)^2
\end{equation}
has a unique minimizer, which is called the \emph{asymptotic center}
of $\{x_n\}_n$, see \cite[p.~58]{b2014}.
A sequence $\{x_n\}_n$ \emph{weakly converges} to a point $x \in \HH$ if it is bounded 
and $x$ is the asymptotic center of each subsequence of $\{x_n\}_n$, see  \cite[p.~103]{b2014}. 
We write $x_n \weakto x$. Note that $x_n \to x$ implies $x_n \weakto x$.
Then the definition of Mosco convergence given in the previous section carries over to functions defined on Hadamard spaces.
%

One of our results (Theorem \ref{theo:final}) uses $\Gamma$-convergence in an intermediate step.
 In the presence of~\eqref{enum:moreau-equiMinorization},
$\Gamma$-convergence has a characterization as equality of upper and lower $\Gamma$-limits,
\begin{equation*}
	 \sup_{\lambda > 0} \liminf_{n\to\infty} f_{n,\lambda} (x)	
	= \sup_{\lambda > 0} \limsup_{n\to\infty} f_{n,\lambda} (x),
\end{equation*}
for all $x \in \HH$, see~\cite[Theorem~2.65]{Attouch84}. 
We will need only the following consequence thereof \cite[Corollary~2.67, (2.166)]{Attouch84}.
 
\begin{theorem} \label{thm:metricgamma}
	Let $(X,d)$ be a metric space and $\{f_n\}_n$ a sequence of functions $f_n\colon X \to (-\infty,+\infty].$ 
	Assume there exist $x_0 \in X$ and $r > 0$ such that~\eqref{enum:moreau-equiMinorization} 
	is satisfied for each $x\in X$ and $n\in\mN.$ If for every $\lambda>0$ and $x\in X$ 
	there exists a number $f(\lambda,x)\in\mR$ such that $f_{n,\lambda}(x)\to f(\lambda,x),$ then $f_n$ is $\Gamma$-convergent and
\begin{equation*}
 f_n \gammaTo \sup_{\lambda>0} f(\lambda,\cdot) \quad {\rm as} \;  n\to\infty.
\end{equation*}
\end{theorem}

%
 The theorem still holds if we consider a decreasing sequence $\{\lambda_k\}_{k \in \mN}$ which converges to 0
instead of all $\lambda > 0$.

\section{Characterization of Mosco Convergence by Moreau Envelopes} \label{sec:result}
%
The implication i) $\implies$ ii) in Theorem \ref{thm:Attouch-MoreauMosco} 
has been generalized to Hadamard spaces in \cite[Theorem 4.1]{bac15}:
\begin{theorem} \label{thm:mosco2moreau}
Let $\HH$ be a Hadamard space
and  $\{ f_n \}_n$ a sequence of functions $f_n \in \Gamma_0(\HH)$
which converges to a proper function
$f\fun$ in the sense of Mosco. Then we have 
\begin{equation*} 
f_{n,\lambda} (x) \to f_\lambda(x),\quad\text{and}\quad J_\lambda^n x \to J_\lambda x, \qquad \text{as } n \to \infty,
\end{equation*}
for every $\lambda >0$ and $x \in \HH.$
\end{theorem}

Note that like in Hilbert spaces $\Gamma$-convergence on Hadamard spaces (and thus Mosco convergence) preserves convexity 
and the above limit function is lsc.
 
Our main result is the inverse implication.
\begin{theorem}  \label{thm:moreau2mosco}
Let $\HH$ be a Hadamard space and $\{ f_n \}_{n}\subset \Gamma_0(\HH).$ Assume that for each $\lambda >0$ the sequence of Moreau envelopes 
$\{ f_{n,\lambda} \}_n$ converges pointwise to the Moreau envelope $f_\lambda$ of a function $f \in \Gamma_0(\HH).$ Then $f_n \moscoTo f$ as $n \to \infty.$
\end{theorem}

Note that, for proper convex lsc functions $f_n$, the pointwise convergence $f_{n,\lambda} \to f(\lambda,\cdot)$ to some limit function $f(\lambda,\cdot)$, does not imply that $f(\lambda,\cdot)$ is a Moreau envelope of some proper convex lsc function; see~\cite[Remark 2.71]{Attouch84}.

\begin{proof}%
Observe that $f(x) \ge f_\lambda(x)\ge f\left(J_\lambda x\right).$ For $x\in\overline{\dom f},$ 
it holds by \cite[Proposition 2.2.26]{b2014} that $\lim_{\lambda \to +0} J_\lambda x = x$
and hence the lower semicontinuity of~$f$ implies
\begin{equation} \label{eq:step0}
 f(x) = \lim_{\lambda\to +0}  f_\lambda(x)= \lim_{\lambda\to +0}  f\left(J_\lambda x\right).		
\end{equation}

\textit{Step~1 (Limsup Inequality).}
 Let us show that, given $x\in\HH,$ there exists a sequence $y_n\to x$ 
 with $\limsup_{n \to \infty} f_n(y_n)\le f(x).$  If $f(x)=\infty,$ then there is nothing to prove. Assume therefore $x\in\dom f.$ Together with the assumption that $f_{n,\lambda}(x) \to f_\lambda(x)$ as $n \to \infty$, we obtain
\begin{equation*}
 f(x) = \lim_{\lambda \to +0} f_\lambda (x)
      = \lim_{\lambda \to +0} \lim_{n\to\infty} f_{n,\lambda}(x).
\end{equation*}
By the diagonalization lemma \cite[Corollary~1.18]{Attouch84}	
there exists a sequence $\{ \lambda_n\}_n$ with 	$\lim_{n\to \infty} \lambda_n	= 0$
such that 
\begin{align}
 f(x) &= \lim_{n\to\infty} f_{n,\lambda_n} (x) 	\notag \\
      &= \lim_{n\to\infty}     
        \left( f_n \left( J_{\lambda_n}^n x \right)
               + \tfrac{1}{2\lambda_n} d \left(x, J_{\lambda_n}^n x \right)^2
        \right),		\label{eq:moreauLn-conv-to-f} 
\end{align} 
where we denote $J_{\lambda}^n \colon= J_{\lambda}^{f_n}.$ Hence $f(x)\geq \limsup_{n \to \infty} f_n \left(J_{\lambda_n}^n x\right).$ We put $y_n := J_{\lambda_n}^n x$ and show that $y_n\to x.$ Indeed, inserting~\eqref{enum:moreau-equiMinorization}
into~\eqref{eq:moreauLn-conv-to-f}, we have
\begin{equation*}
  f(x) \geq \limsup_{n\to\infty}
    \left( \bigl( \tfrac{1}{2\lambda_n} - r \bigr) d(x,y_n)^2 - r \right)
\end{equation*}
and we can conclude that $y_n\to x.$

\textit{Step 2.} Let us show that $J_\lambda^n x\to J_\lambda x.$ 
From the previous step, we know that there exists a sequence 
$y_n\to J_\lambda x$ with 
$\limsup_{n \to \infty} f_n(y_n)\le f\left(J_\lambda x\right).$ 
Then we obtain
\begin{equation*}
f_\lambda(x) = f\left(J_\lambda x\right)+\tfrac{1}{2\lambda}d\left(x,J_\lambda x\right)^2  
\geq \limsup_{n\to\infty} \left( f_n(y_n)+\tfrac{1}{2\lambda} d\left(x,y_n\right)^2  \right)
\end{equation*}
and by employing~\eqref{eq:ineq} we arrive at
\begin{equation*}
f_\lambda(x)  \geq \limsup_{n\to\infty} \left( f_{n,\lambda}(x) +\tfrac{1}{2\lambda}d\left(J_\lambda^n x,y_n\right)^2  \right) 
= f_\lambda (x)+\limsup_{n\to\infty} \tfrac{1}{2\lambda} d\left(J_\lambda^n x,y_n\right)^2.
\end{equation*}
Hence we conclude $J_\lambda^n x\to J_\lambda x.$
 
\textit{Step 3 (Liminf Inequality).} 
Let $x_n\weakto x.$ We have to prove $\liminf_{n \to \infty} f_n(x_n)\geq f(x).$ 
By the definition of the Moreau envelope and \eqref{eq:ineq} we have
 \begin{align*}
  f_n(x_n) 
	& \geq f_n\left(J_\lambda^n x_n\right)
    +\tfrac1{2\lambda} d\left(x_n,J_\lambda^n x_n\right)^2 \nonumber \\
  & \geq f_n\left(J_\lambda^n x\right)+\tfrac1{2\lambda}d\left(x,J_\lambda^n x\right)^2
    +\tfrac1{2\lambda}d\left(J_\lambda^n x_n,J_\lambda^n x\right)^2
    +\tfrac1{2\lambda} d\left(x_n,J_\lambda^n x_n\right)^2
    -\tfrac1{2\lambda}d\left(x,J_\lambda^n x_n\right)^2. 
	\end{align*}	
	By the nonpositive curvature inequality in \eqref{eq:reshet} we obtain
	\begin{align}
  f_n(x_n)
  & \geq 
	f_n\left(J_\lambda^n x\right)+\tfrac1{2\lambda}d\left(J_\lambda^n x,x_n\right)^2
    -\tfrac1{2\lambda}d\left(x,x_n\right)^2. \label{eq:moscoHad-liminf}
 \end{align}
Let us show that $f_n\left(J_\lambda^n x\right)$ converges as $n \to \infty$.
Consider
\begin{alignat*}{10}
 f_{n,\lambda}(x) &= f_n(J_\lambda^n x) & &+ \trez{2\lambda} \di(x, J_\lambda^n x)^2,\\
 f_{\lambda}(x)   &= f(J_\lambda x)     & &+ \trez{2\lambda} \di(x, J_\lambda x)^2.
\end{alignat*}
By assumption we have $f_{n,\lambda} (x) \to f_\lambda (x)$,
and by Step~2 also $J_\lambda^n x \to J_\lambda x$ as $n\to\infty$. This implies
\begin{equation}		\label{eq:atProxConv}
  f_n(J_\lambda^n x)\to f(J_\lambda x). 
\end{equation}
By the definition of the weak limit of $\{x_n\}_n$,
for every subsequence $n_k \to \infty$, we have 
\begin{equation*} 
\limsup_{k\to\infty} \left(\di(J_\lambda x, x_{n_k} )^2  - \di(x, x_{n_k})^2\right)\ge\limsup_{k\to\infty} \di(J_\lambda x, x_{n_k} )^2  -\limsup_{k\to\infty} \di(x, x_{n_k})^2\ge0.
\end{equation*}
As the subsequence was arbitrary, together with Step~2, this implies
\begin{equation*}		\label{eq:moscoHad-liminf-di}
  \liminf_{n \to\infty} \left( \di(J_\lambda^n x, x_{n} )^2  - \di(x, x_{n})^2 \right)
  \geq 0.
\end{equation*}
Returning to~\eqref{eq:moscoHad-liminf}, the previous inequality
and~\eqref{eq:atProxConv} yield
\begin{equation*}
 \liminf_{n \to \infty} f_n(x_n) \geq f(J_\lambda x).
\end{equation*}
If $x \in \overline{\dom f}$, then from \eqref{eq:step0} we obtain
\begin{equation*}
 \liminf_{n\to\infty} f_n(x_n)\ge f(x).		
\end{equation*}
For $x \not \in \overline{\dom f}$ we can repeat the above conclusions
for the finite continuous convex functions $g_n := f_{n,\mu}$ and $g = f_\mu$ for some fixed $\mu >0$
instead of $f_n$ and $f$.
Note that the assumptions are fulfilled by the semigroup property of the Moreau envelopes.
Finally we let $\mu \to + 0$ and invoke \eqref{conv_lambda}.
This concludes the proof. 
\end{proof}

\begin{remark} \label{lambdak}
One can easily check that Theorem~\ref {thm:moreau2mosco} remains true if the assumption 
is made just for a sequence $\{ \lambda_k\}_{k \in \mN}$ of positive numbers which
strictly decreases to $0$, instead of all $\lambda >0$. This observation will be important in the proof of Theorem~\ref{theo:final}.
\end{remark}

Recall that a sequence of convex closed sets $C_n\subset\HH$ converges to a convex closed set $C\subset\HH$ 
in the sense of Frol\'{i}k-Wijsman if the respective distance functions converge pointwise; that is, 
if $d(x, C_n)\to d(x,C)$ for each $x\in\HH.$ This concept originated in \cite{frolik,wijsman}. 
On the other hand, a sequence of convex closed sets $C_n\subset\HH$ converges to a convex closed set $C\subset\HH$ 
in the sense of Mosco if the indicator functions $\iota_{C_n}$ converge in the sense of Mosco to the indicator function~$\iota_C.$ The following is a~direct consequence of our main result.
\begin{corollary}[Frol\'{i}k-Wijsman convergence] \label{cor:sets}
A sequence of convex closed sets $C_n\subset\HH$ converges to a convex closed set $C\subset\HH$ 
in the sense of Frol\'{i}k-Wijsman if and only if it converges to $C$ in the sense of Mosco.
\end{corollary}
\begin{proof}
 Observe that the Moreau envelope of $\iota_C$ with $\lambda=\frac12$ is precisely the distance function squared $d(\cdot,C)^2$ 
 and apply Theorems \ref{thm:mosco2moreau} and \ref{thm:moreau2mosco}.
\end{proof}

\section{Topology of the Mosco convergence} \label{sec:topo}
Let $\left(\HH,d\right)$ be a \emph{separable} Hadamard space and  $\Gamma_0(\HH)$  
the convex cone of proper convex lsc functions on $\HH$.
We will show that there exists a complete metric $\rho$ on $\Gamma_0(\HH)$ such that 
\begin{equation} \label{eq:moscometric}
 f_n\moscoTo f \quad \text{if and only if } \rho\left(f_n,f\right)\to0
\end{equation}
as $n \to \infty$.
The Hilbert space case of this result was proved in \cite[Theorem 3.36]{Attouch84}. 
Let $\{ \lambda_k\}_{k \in \mN}$ be a~sequence of positive numbers which
strictly decreases to $0.$ Following~\cite[Section 3.5.1]{Attouch84} we consider the family of pseudometrics
\begin{equation} \label{gamma-topo}
 e_{\lambda_k,x}(f,g) := \left| f_{\lambda_k}(x) - g_{\lambda_k}(x) \right|,
\end{equation}
indexed by $x\in\HH$ and $k \in \mN$.
By the equivalence of Mosco convergence with the convergence of Moreau envelopes 
shown in the previous section, 
the above family of pseudometrics gives a~topology corresponding to Mosco convergence 
on $\Gamma_0(\HH)$.
We will show that the space $\Gamma_0(\HH)$ together with the topology 
induced by the family~\eqref{gamma-topo} is metrizable.
The proof relies on the equi-Lipschitz property of the sequences $\{f_{n,\lambda}\}_n$ 
which is the content of the following lemma.
\begin{lemma}\label{thm:moreau-conv-and-lipschitz}
Let $\{f_n\}_{n}\subset\Gamma_0(\HH)$ and suppose that there exist $0 < \lambda_0 < \lambda_1,$ a point $x_0 \in \HH$ and $f_0, f_1 \in \mR$ such that the Moreau envelopes fulfill
\begin{equation*}
 \lim_{n\to\infty} f_{n,{\lambda_i}}(x_0) = f_i,
\end{equation*}
for~$i=0,1.$ Then for each $\lambda\in\left(0,\lambda_0\right)$ and $R>0$ there exists a Lipschitz constant $L>0$ such that
\begin{equation*}
\left|f_{n,\lambda}(x) - f_{n,\lambda}(y)\right|< L d(x, y),
\end{equation*}
for every $x, y \in B_R(x_0)$ and $n\in\mN.$
\end{lemma}
\begin{proof} 
Denote by $J_\lambda^n$ the proximal mapping of~$f_n.$ 

Step 1: By assumption we have for $n_0 \in \mN$ sufficiently large and each $n \ge n_0$ that $|f_{n,\lambda_i} (x_0) - f_i| \le 1$ for $i=0,1.$ Hence we get
\begin{equation*}
f_n \left(J^n_{\lambda_0} x_0 \right) + \trez{2\lambda_0} d\left(J^n_{\lambda_0} x_0,x_0 \right)^2 \le f_0 + 1
\end{equation*}
and 
\begin{equation*}
f_n \left(J^n_{\lambda_0}  x_0 \right) + \trez{2\lambda_1}d \left(J^n_{\lambda_0} x_0,x_0 \right)^2 \ge f_{n,\lambda_1} (x_0) \ge f_1 - 1.
\end{equation*}
Summing up the above inequalities yields
\begin{equation*}
 \bigl( \trez{2\lambda_0}  - \trez{2\lambda_1} \bigr) d \left(J^n_{\lambda_0} x_0,x_0 \right)^2 \le 
2 + f_0 - f_1,
\end{equation*}
for every $n \ge n_0.$ Observe then that $d \left(J^n_{\lambda_0} x_0,x_0 \right)$ is bounded and therefore, given $R > 0,$ there exists a constant $C>0$ such that $d \left(J^n_{\lambda} x,x\right)\le d \left(J^n_{\lambda_0} x,x\right)\le C$ for every $x \in B_{R}(x_0), \lambda\in\left(0,\lambda_0\right),$ and $n\in\mN.$
Here we used the simple fact that $\lambda \mapsto \di( J^n_\lambda x, x)$ is increasing,
see e.g.~\cite[Proof of Theorem~2.2.25]{b2014}.

Step 2: The proof follows the lines of Lemma~\ref{lem:continuity-M-Y}, where only a single function is considered instead of a sequence of functions. Choose  $x,y \in B_{R}(x_0)$ and $\lambda\in\left(0,\lambda_0\right).$ Then by the definition of envelopes, for each $n\in\mN,$ we have
\begin{align*}
 f_{n,\lambda}(y) & \le f_n\left( J^n_{\lambda}x\right)+\trez{2\lambda}d\left(y,J^n_{\lambda}x\right)^2 \\ 
&\le f_{n,\lambda}(x) -\trez{2\lambda}d\left(x,J^n_{\lambda}x\right)^2+\trez{2\lambda}d\left(y,J^n_{\lambda}x\right)^2 \\  
&=  f_{n,\lambda}(x)
+\trez{2\lambda} \bigl( \di(y, J^n_{\lambda}x) - \di(x, J^n_{\lambda}x)\bigr)
                            \bigl(\di(y, J^n_{\lambda}x) + \di(x, J^n_{\lambda}x)\bigr)\\
&\le f_{n,\lambda}(x) +\trez{2\lambda}d(x,y)\left(2d\left(x,J^n_{\lambda}x\right)+d(x,y)\right) \\  
&\le f_{n,\lambda}(x) +\frac{C+R}{\lambda}d(x,y),
\end{align*}
where we used the triangle inequality for the third inequality.
Interchanging $x$ and $y$ yields
\begin{equation*}
 \left| f_{n,\lambda}(x)-f_{n,\lambda}(y)\right| \le \frac{C+R}{\lambda}d(x,y),
\end{equation*}
which finishes the proof.
\end{proof}

Equi-Lipschitz sequences have the following property.
\begin{remark} \label{rem:equilipschitz}
Let $X$ and $Y$ be metric spaces and $\{g_n\}_n$ be a sequence of equi-Lipschitz mappings $g_n\colon X\to Y$ with Lipschitz constant $L > 0.$ Furthermore, let $Z\subset X$ be a dense set such that $\lim_{n\to\infty} g_n(z) = G(z)$ exists for every $z \in Z$. Then the limit
\begin{equation*} 
	\lim_{n\to\infty} g_n(x) = G(x)
\end{equation*}
exists for each $x \in X$ and the mapping $G\colon X\to Y$ is also $L$-Lipschitz. 
\end{remark}



The proof of the following lemma follows the lines of \cite[Theorem 3.36]{Attouch84}.
\begin{lemma} \label{prop:metrics_1}
The space $\Gamma_0(\HH)$ together with the topology induced by the family~\eqref{gamma-topo} is metrizable.
\end{lemma}
\begin{proof}
We show that the family of pseudometrics~\eqref{gamma-topo} can be replaced by a \emph{countable} family 
which induces the same topology. To this end, we choose a countable dense set $\left\{ x_l\right\}_{l \in \mN} \subset\HH$. 
Let $\{f_n\}_n\subset \Gamma_0(\HH).$ We will show that if
\begin{equation} \label{discrete}
 f\left(\lambda_k,x_l\right):=\lim_{n\to\infty}f_{n,\lambda_k}\left( x_l \right)
\end{equation}
for every $k,l\in\mN,$ then, for every $k\in\mN,$ there exists a unique continuous function $f\left(\lambda_k,\cdot\right)$ such that
\begin{equation} \label{toshow}
 f\left(\lambda_k,x\right)=\lim_{n\to\infty}f_{n,\lambda_k}\left(x\right),
\end{equation}
for every $x\in\HH$. Indeed, for each fixed $\lambda_k,$ we obtain by Lemma~\ref{thm:moreau-conv-and-lipschitz} that
the functions $\left\{f_{n,\lambda_k}\right\}_n$ are locally equi-Lipschitz on $\HH$. 
By Remark~\ref{rem:equilipschitz}  this implies for each $\lambda_k$
that 
\begin{equation*}
 \lim_{n\to\infty}f_{n,\lambda_k}\left(x\right) =\phi(\lambda_k,x),\qquad x\in\HH,
\end{equation*}
where $\phi(\lambda_k,\cdot)$ is a locally Lipschitz function. 
By \eqref{discrete} we have $\phi\left(\lambda_k,x_l\right)=f\left(\lambda_k,x_l\right)$ 
and since $\{x_l\}_l$ is dense in $\HH$ we can put 
$f\left(\lambda_k,x\right) := \phi\left(\lambda_k,x\right)$ for every $x\in\HH$ which gives \eqref{toshow}.

We hence obtain that
\begin{equation} \label{metric_1}
\rho_e(f,g) := \sum_{k,l \in \mN^2} \frac{1}{2^{k+l}} \frac{e_{\lambda_k,x_l}(f,g)}{1+e_{\lambda_k,x_l}(f,g)}
\end{equation}
is a metric on $\Gamma_0(\HH)$.
\end{proof}

Unfortunately, since $\left\{f_{n,\lambda_k}\right\}_n$ does not in general converge to the
Moreau envelope of a~function from $\Gamma_0(\HH)$, 
the metric \eqref{metric_1} is not complete. This was observed already by Attouch \cite[p.\ 325]{Attouch84} along with a modification of the metric which results in a complete metric on $\Gamma_0(H),$ for $H$ a Hilbert space. Whereas his approach seems to work in linear spaces only, in Hadamard spaces one can instead define an updated family of pseudometrics by
\begin{equation} \label{metric_complete}
e_{{\lambda_k},x}(f,g) := \left|f_{\lambda_k}(x)-g_{\lambda_k}(x) \right| 
\quad {\rm and} \quad 
r_{\lambda,x}(f,g) :=d \left(J_{\lambda_k}^f(x),J_{\lambda_k} ^g(x)\right),
\end{equation}
for every $x\in\HH$ and $k \in \mN$. First observe that the family of pseudometrics \eqref{metric_complete} also induces the topology corresponding to Mosco convergence:
Indeed the family \eqref{metric_complete} induces a topology at least as strong
as the topology induced by the family \eqref{gamma-topo} and therefore at least as strong
as the Mosco topology, by the previous considerations. 
On the other hand we have by Theorem~\ref{thm:mosco2moreau}
that Mosco convergence implies convergence with respect to the pseudometrics in~\eqref{metric_complete}.
%
\begin{theorem} \label{theo:final}

The space $\Gamma_0(\HH)$ together with the Mosco topology admits a complete metric.
\end{theorem}

\begin{proof}
For a dense set $\{x_l\}_{l \in \mN}$ of $\HH$ we consider the countable family
\begin{equation*}
 e_{{\lambda_k},x_l}(f,g) := \left|f_{\lambda_k}(x_l)-g_{\lambda_k}(x_l) \right| 
\quad {\rm and} \quad 
r_{{\lambda_k},x_l}(f,g) :=d \left(J_{\lambda_k}^f(x_l),J_{\lambda_k} ^g(x_l)\right)
\end{equation*}
for all $l,k \in \mN$. 
This family is nested between the family of pseudometrics~\eqref{metric_complete}
and the subfamily $\{e_{\lambda_k, x_l}\}_{k,l}$ defined in Lemma~\ref{prop:metrics_1},
which both induce the Mosco topology. Therefore this family induces the Mosco topology, too.

Define $\rho$ as in \eqref{metric_1} with the additional summands due to $r_{{\lambda_k},x_l}(f,g)$, where $k,l \in \mN$. Then~$\rho$ is a~metric on~$\Gamma_0(\HH);$ see Lemma~\ref{prop:metrics_1}. Let $\{f_n\}_n\subset\Gamma_0(\HH)$ be a Cauchy sequence with respect to $\rho$ and denote the corresponding proximal mappings by~$J_\lambda^n.$ Then, for every $k,l \in \mN$, there exist $\phi(\lambda_k,x_l)\in\mR$ and $\Phi(\lambda_k,x_l)\in\HH$ such that
\begin{equation*}
f_{n,\lambda_k} (x_l) \to \phi(\lambda_k,x_l), \quad\text{and}\quad J^n_{\lambda_k} (x_l) \to \Phi(\lambda_k,x_l),
\end{equation*}
as $n\to\infty.$ The functions $f_{n,\lambda_k}\colon\HH\to\mR$ are locally equi-Lipschitz (Lemma~\ref{thm:moreau-conv-and-lipschitz})
and the mappings $J^n_{\lambda_k}\colon \HH \to \HH$
are $1$-Lipschitz.
Therefore, by Remark~\ref{rem:equilipschitz}, the limits
\begin{equation} \label{eq:limpoints}
f_{n,\lambda_k} (x) \to \phi(\lambda_k,x), \quad\text{and}\quad J^n_{\lambda_k} (x) \to \Phi(\lambda_k,x),\qquad \text{as }n\to\infty,
\end{equation}
exist for every $k \in \mN$ and $x \in \HH$. Observe that the functions $\phi(\lambda_k,\cdot)$ are convex and locally Lipschitz. Define
\begin{equation} \label{limit-function}
 f(x) := \sup_{k \in \mN} \phi(\lambda_k, x),
\end{equation}
which is a convex lsc function. By Lemma~\ref{thm:moreau-equiMinorization}, the functions $f_n$ satisfies~\eqref{enum:moreau-equiMinorization} for some $x_0 \in \HH$ and $r > 0,$ which along with~\eqref{eq:limpoints} enables to invoke Theorem~\ref{thm:metricgamma} and conclude that $f_n \gammaTo f;$ as was observed already in \cite[Remark 2.71]{Attouch84}.

Next we show that
\begin{equation} \label{eq:envconv}
\lim_{n\to\infty} f_{n,\lambda_k}(x)= f_{\lambda_k}(x),
\end{equation}
for each $k\in\mN$ and all $x \in \HH.$ Using that $d(x,\cdot)^2$ is a continuous function on $\HH$, it follows directly from 
\cite[Theorem 2.15]{Attouch84} that
\begin{equation}\label{eq:gn-conv}
f_n + \trez{2{\lambda_k}} d(x, \cdot)^2  \quad \gammaTo \quad f + \trez{2{\lambda_k}} d(x,\cdot)^2,	
\end{equation}
for each $k \in \mN.$
By property~\cite[Theorem~1.10]{Attouch84} of  $\Gamma$-convergence
the limit $\Phi(\lambda_k, x)$ in~\eqref{eq:limpoints} is a minimizer of the right-hand side
of~\eqref{eq:gn-conv},
and the minimal values $f_{n,\lambda_k}(x)$
of the left-hand sides of~\eqref{eq:gn-conv} converge to the minimal value $f_{\lambda_k}(x)$
of the right-hand side, which implies~\eqref{eq:envconv}.


On account of Theorem~\ref{thm:moreau2mosco} and Remark~\ref{lambdak}, we conclude that $f_n\moscoTo f.$ 
Furthermore,
\begin{equation*}
	f\left(\Phi(\lambda_k,x)\right) \le \liminf_{n\to\infty} f_{n,\lambda_k}(x) 
 	= \phi(\lambda_k, x) < \infty,
\end{equation*}
and thus in particular $\dom f \neq\emptyset$.
Hence $f\in\Gamma_0(\HH)$ and the proof is complete.
\end{proof}

\subsection*{Acknowledgments} 
We would like to express our gratitude to all three referees for their useful comments. We are especially indebted 
to the referee who brought to our attention \cite[Theorem 3.36]{Attouch84} and suggested we prove its Hadamard space version.
\\
Funding by the German Research Foundation (DFG) within the project STE 571/13-1 is gratefully acknowledged.

\bibliographystyle{abbrv}
\bibliography{moscom}
\end{document}